\RequirePackage[l2tabu, orthodox]{nag}
\RequirePackage{fixltx2e}
\documentclass[10pt,a4paper,reqno,oneside,final]{smfart}

\usepackage{leftidx}
\usepackage{appendix}
\usepackage{tikz}\usetikzlibrary{decorations.markings,matrix,arrows}
\usepackage{amsfonts}
\usepackage{amsthm}
\usepackage[T1]{fontenc}
\usepackage[mathscr]{eucal}
\usepackage{setspace}
\usepackage{amssymb,latexsym,amsmath,amscd}
\usepackage{graphicx}
\usepackage{showkeys}
\usepackage[french, english]{babel}
\usepackage{layout}
\usepackage{enumerate}
\usepackage{indentfirst}
\usepackage[varg]{pxfonts}
\usepackage[stretch=10]{microtype}
\usepackage{booktabs}
\usepackage{xspace}
\usepackage{eufrak}
\usepackage[colorlinks=false, pdfborder={0 0 0}]{hyperref} 
\usepackage{nameref}
\usepackage{cleveref}
\usepackage{calrsfs}
\usepackage{filecontents}
\usepackage{mathtools}
\usepackage{titlesec}
\usepackage{fixltx2e}
\usepackage{todonotes}

\titleformat{\paragraph}
{\normalfont\normalsize\bfseries}{\theparagraph}{1em}{}
\titlespacing*{\paragraph}
{0pt}{3.25ex plus 1ex minus .2ex}{1.5ex plus .2ex}

\theoremstyle{plain}
\newtheorem{thm}{Theorem}[section]

\newtheorem*{thm*}{Theorem}

\newtheorem{lem}[thm]{Lemma}

\newtheorem{pro}[thm]{Proposition}
\newtheorem{pro-def}[thm]{Proposition-Definition}
\newtheorem{cor}[thm]{Corollary}
\newtheorem{conj}[thm]{Conjecture}

\theoremstyle{definition}

\theoremstyle{remark}

\newcommand{\bC}{\mathbf{C}}

\newcommand{\bR}{\mathbf{R}}

\newcommand{\bQ}{\mathbf{Q}}

\newcommand{\bZ}{\mathbf{Z}}

\newcommand{\gO}{\Omega}
\newcommand{\ga}{\alpha}
\newcommand{\gS}{\Sigma}
\newcommand{\gb}{\beta}

\newcommand{\go}{\omega}

\newcommand{\cH}{\mathcal{H}}

\newcommand{\cA}{\mathcal{A}}

\newcommand{\cO}{\mathcal{O}}

\newcommand{\colonec}{\mathrel{:=}}
\newcommand{\bss}{\backslash}

\newcommand{\Alb}{\mathrm{Alb}}
\newcommand{\alb}{\mathrm{alb}}

\newcommand{\Hol}{\mathrm{Hol}}
\newcommand{\tors}{\mathrm{tors}}
\newcommand{\SU}{\mathrm{SU}}
\newcommand{\Sp}{\mathrm{Sp}}

\newcommand{\CH}{\mathrm{CH}}

\newcommand{\NS}{\mathrm{NS}}

\newcommand{\Sym}{\mathrm{Sym}}

\newcommand{\Pic}{\mathrm{Pic}}
\newcommand{\rank}{\mathrm{rank}}

\newcommand{\cupp}{\mathbin{\smile}}

\newcommand{\dto}{\dashrightarrow}

\newcommand{\vast}{\bBigg@{4}}
\newcommand{\Vast}{\bBigg@{5}}

\tikzset{node distance=2cm, auto}

\numberwithin{equation}{section}

\title{Lagrangian constant cycle subvarieties in Lagrangian fibrations} 

\author{Hsueh-Yung Lin}

\address{
Mathmatisches Institut \\
Universität Bonn \\
Endenicher Allee 60\\
53115 Bonn, Germany.}
 \email{hsuehyung.lin.math@gmail.com}

\begin{document}
\selectlanguage{english}

\begin{abstract}
We show that the image of a dominant meromorphic map from an irreducible compact Calabi-Yau manifold $X$  whose general fiber is of dimension strictly between $0$ and $\dim X$ is rationally connected. Using this result, we construct for any hyper-Kähler manifold $X$ admitting a Lagrangian fibration a Lagrangian constant cycle subvariety $\gS_H$ in $X$ which depends on a divisor class $H$ whose restriction to some smooth Lagrangian fiber is ample. If $\dim X = 4$, we also show that up to a scalar multiple, the class of a zero-cycle supported on $\gS_H$ in $\CH_0(X)$ does not depend neither on $H$ nor on the Lagrangian fibration (provided $b_2(X) \ge 8$).
\end{abstract}

\maketitle

\section{Introduction}

This note is devoted to the construction of some subvarieties $Y$ in a projective hyper-Kähler manifold $X$   admitting a Lagrangian fibration such that all points in $Y$ are rationally equivalent in $X$. These subvarieties, called constant cycle subvarieties in~\cite{HuybCCC}, depend on the choice of a Lagrangian fibration and a divisor class $H \in \Pic(X)$ whose restriction to some smooth Lagrangian fiber is ample, thus the image of the Gysin map $\CH_0(Y) \to \CH_0(X)$   depends \emph{a priori} on these choices as well. The result of this note is motivated by the conjectural picture on the splitting property of the conjectural Bloch-Beilinson filtration of projective hyper-Kähler manifolds due to Beauville~\cite{BeauvilleSplitBB} and Voisin~\cite{VoisinHKcoisotp} as we will explain below.

The study of constant cycle subvarieties in hyper-Kähler manifolds was initiated by Huybrechts in the case of K3 surfaces~\cite{HuybCCC}.
Our motivation for studying these subvarieties comes rather from the attempt to generalize the Beauville-Voisin canonical zero-cycle of a projective K3 surface~\cite{BV} to higher dimensional cases. For a K3 surface $S$, recall that there are at least two ways to characterize the canonical zero-cycle  $o_S$:
\begin{enumerate}[i)]
\item $o_S$ is the degree-one generator of the image of the intersection product~\cite{BV} 
$$\cupp : \CH^1(S) \otimes \CH^1(S) \to \CH^2(S);$$
\item $o_S$ is the class of a point supported on a constant cycle curve in $S$~\cite[Lemma 2.2]{Voisin0cycleK3}. 
\end{enumerate} 

Each characterization gives \emph{a priori} different generalization of $o_S$. The first one is related to Beauville's conjecture on the weak splitting property of the Chow ring of projective hyper-Kähler manifolds~\cite{BeauvilleSplitBB}:
\begin{conj}[Beauville~\cite{BeauvilleSplitBB}]\label{conj-Beauville}
Let $X$ be a projective hyper-Kähler manifold. The restriction of the cycle class map $\CH^\bullet(X)_\bQ \colonec \CH^\bullet(X) \otimes_\bZ \bQ \to H^\bullet(X,\bQ)$ to the $\bQ$-sub-algebra generated by divisor classes is injective.
\end{conj}  
The reader is referred to~\cite{BeauvilleSplitBB, VoisinBVconjHilbK3, FuBVKummer, RiessBconjLF, VoisinHKcoisotp} for recent developments of this conjecture. In particular, since $H^{4n}(X,\bQ) = \bQ$, Beauville's conjecture contains as a sub-conjecture the statement that the intersection of any $2n$ divisor classes in $\CH^\bullet(X)_\bQ$ is proportional to the same degree one zero-cycle $o_X \in \CH^{2n}(X)_\bQ$ where $2n$ is the dimension of $X$, which generalizes property $i)$ of $o_S$.

The generalization of property $ii)$ is formulated in~\cite[Conjectures  0.4 and 0.8]{VoisinHKcoisotp}: for $0 \le i \le n$, let $S_i \CH_0(X)$ denote the subgroup of $\CH_0(X)$ generated by the classes of points whose rational orbit is of dimension $\ge i$\footnote{The \emph{rational orbit} of a point $z \in X$ is the set $O_z \subset X$ of points in $X$ which are rationally equivalent to $z$. The set $O_z$ is a countable union of Zariski closed subsets of $X$ and we define $\dim O_z$ to be the maximum of the dimension of all irreducible components of $O_z$.}. One hopes that this decreasing  filtration $S_\bullet \CH_0(X)$ would define a splitting of the conjectural Bloch-Beilinson filtration $F^\bullet_{BB}$ in the sense that the inclusion $S_i\CH_0(X) \hookrightarrow \CH_0(X)$ induces an isomorphism
$$S_i\CH_0(X) \xrightarrow{\sim} \CH_0(X)/F^{2n-2i+1}_{BB}\CH_0(X).$$

Using the axioms of the Bloch-Beilinson conjecture, the surjectivity of the above map is proven in~\cite{VoisinHKcoisotp} to be predicted by the following conjecture:
\begin{conj}[\cite{VoisinHKcoisotp}]\label{conjccs}
Let $X$ be a projective hyper-Kähler manifold of dimension $2n$. For any $0 \le i \le n$, the dimension of the set of points $S_iX \subset X$ whose rational orbit has dimension $\ge i$ is $2n-i$.
\end{conj}
We refer to~\cite{VoisinHKcoisotp} for more details on Voisin's circle of ideas for studying the splitting property of the Bloch-Beilinson filtration on $\CH_0(X)$. Note that by~\cite[Theorem 1.3]{VoisinHKcoisotp}, we always have $S_iX \le 2n-i$. So when $i = n$, Conjecture~\ref{conjccs} is equivalent to the existence of constant cycle subvarieties of $X$ of dimension $n$ (which are necessarily Lagrangian, by Roitman-Mumford's theorem~\cite[Proposition 10.24]{VoisinII}) and one would expect to recover the conjectural canonical zero-cycle for any projective hyper-Kähler manifold $X$ by taking the class of a point in any of these constant cycle Lagrangian subvarieties, hence the second generalization of $o_S$.

In general it is difficult to construct Lagrangian constant cycle subvarieties. However, if $X$ admits a Lagrangian fibration $\pi : X \to B$ (\emph{i.e.} a surjective holomorphic map to a projective variety $B$ whose general fiber is a connected Lagrangian submanifold), we prove  in Section~\ref{sec-ex} the following theorem, which proves in particular Conjecture~\ref{conjccs} in the case $i = n$ for every Lagrangian fibration. 

\begin{thm}\label{thm-constr}
Let $X$ be a projective hyper-Kähler manifold admitting a Lagrangian fibration $\pi : X \to B$. For each divisor class $H \in \Pic(X)$ whose restriction to some smooth Lagrangian fiber is ample, there exists a Lagrangian constant cycle subvariety $\gS_{\pi,H} \subset X$ dominating $B$ and all of whose points are rationally equivalent to $\frac{1}{\deg H^n \cdot [F]} \cdot H^n \cdot [F]$ in $\CH_0(X)$\footnote{Since $h^{1,0}(X) = 0$, by Roitman's theorem~\cite{Roitmantor} $\CH_0(X)$ is torsion free. Thus $\frac{1}{\deg H^n \cdot [F]} \cdot H^n \cdot [F]$ is well-defined.} where $[F]$ the class of a fiber of $F$ of $\pi$.
\end{thm}

The fact that $ [F] \in \CH^n(X)$ is independent of $F$ is a direct consequence of the following general result. 

\begin{thm}\label{thm-RCbase}
Let $X$ be a  Calabi-Yau manifold and $f : X \dto B$ a dominant meromorphic map over a Kähler base $B$. If $0 < \dim B < \dim X$, then $B$ is rationally connected.
\end{thm}

Here a  \emph{Calabi-Yau manifold} is an irreducible (in the sense of Riemannian geometry) compact Kähler manifold with finite fundamental group and trivial canonical bundle. The Riemannian holonomy group of a Calabi-Yau manifold associated to its Kähler metric is either $\SU(n)$ or $\Sp(n)$. Hyper-Kähler manifolds and Calabi-Yau manifolds in the strict sense are examples of Calabi-Yau manifolds. Strictly speaking, we will only apply Theorem~\ref{thm-RCbase} to projective Calabi-Yau manifolds and in this case, Theorem~\ref{thm-RCbase} is already known to be true by~\cite[Theorem 14]{Kollar-Larson}. However our proof is different from the proof of~\cite[Theorem 14]{Kollar-Larson} and works also for non-projective Calabi-Yau manifolds.

Theorem~\ref{thm-RCbase} will allow to rephrase Theorem~\ref{thm-constr} replacing the fiber $F$ by the cycle $L^n \in \CH^n(X)$. In the case where $\dim X = 4$, under the mild assumption that a very general projective deformation of the Lagrangian fibration $\pi : X \to B$ with $\rho(X) \ge 3$ satisfies Matsushita's conjecture\footnote{
We say that a Lagrangian fibration $\pi : X \to B$ satisfies Matsushita's conjecture if either $\pi : X \to B$ is isotrivial or the induced moduli map $B \dto \cA_n$ to some suitable moduli space of abelian varieties is generically injective. } (for instance when $b_2(X) \ge 8$~\cite{VoisinLag}), Theorem~\ref{thm-constr} allows us to define for such a variety $X$ a canonical zero-cycle $o_X \in \CH_0(X)$ by taking the class of a point supported on any Lagrangian constant cycle subvariety $\gS_{\pi,H}$ defined above, whose class is also proportional to the product of any 4 divisors:
 
 \begin{thm}\label{thm-indep} 
 \hfill
 \begin{enumerate}[i)]
 \item  Let $f : X \to B$ be a projective hyper-Kähler manifold admitting a Lagrangian fibration of dimension $2n$. Let $L \colonec f^*c_1(\cO_B(1))$ where $\cO_B(1)$ is an ample line bundle on $B$
 and let $D$ be a divisor in $X$. Assume that $[D]_{|F} \ne 0 \in H^2(F,\bQ)$ for some fiber $F$ of $\pi$ or $\dim X = 4$, then the zero-cycle $L^n \cdot D^n$ is proportional to the class of a point $x \in X$ which belongs to a Lagrangian constant cycle subvariety constructed in Theorem~\ref{thm-constr}.
\item Assume that $\dim X = 4$. If a very general projective deformation of $X$ preserving  the Lagrangian fibration with $\rho(X) \ge 3$ satisfies Matsushita's conjecture (in particular if $b_2(X) \ge 8$~\cite{VoisinLag}), then the class of a point in $\gS_{\pi,H}$ modulo rational equivalence is independent of the divisor class $H$.
\item Under the same hypothesis as in $ii)$, the class of a point in $\gS_{\pi,H}$ modulo rational equivalence is independent of the Lagrangian fibration.
\end{enumerate}
 \end{thm}

\section{Base variety of rationally fibered Calabi-Yau manifolds}\label{sec-RCbase}

 We will prove Theorem~\ref{thm-RCbase} in this section.

\begin{proof}[Proof of Theorem~\ref{thm-RCbase}]

Up to a bimeromorphic modification, we suppose that $B$ is smooth. If $B$ has a non-trivial holomorphic $2$-form $\ga$, then $2 \le \dim B < \dim X$ and $f^*\ga \ne 0$ is degenerate, contradicting the Calabi-Yau assumption. Thus $B$ is projective.

By Graber-Harris-Starr's theorem~\cite{GHS}, it suffices to show that if $B$ satisfies the condition in Theorem~\ref{thm-RCbase}, then $B$ is uniruled. Indeed, suppose that $B$ is not rationally connected and let $B \dto B'$ be the MRC-fibration of $B$, then the composition $X \dto B \dto B'$ is dominant with $0 < \dim B' < \dim X $. So $B'$ would be uniruled, contradicting~\cite[Corollary$1.4$]{GHS}.

Now suppose that $B$ is not uniruled. By~\cite{BDPP}, the canonical class $c_1(K_B)$ is pseudo-effective. This means that the class $c_1(K_B) \in H^2(B,\bR)$ is a limit of effective divisor classes. Let $X \xleftarrow{p} \tilde{X} \xrightarrow{q}  B$ be a resolution of $f : X \dto B$ with $\tilde{X}$ smooth. As $p_*q^*$ maps effective divisor classes to effective divisor classes, the class $p_*q^*c_1(K_B)$ is also pseudo-effective.

Since $q : \tilde{X} \to B$ is surjective, the induced map $q^*K_B \to \gO_{\tilde{X}}^k$ is non-zero where $k = \dim B$. As $X$ is smooth, this map determines a non-zero morphism $L \to \gO_X^{k}$ where $L$ is a line bundle such that $c_1(L) =p_*q^*c_1(K_B)$.

Let $\go$ be a Kähler form in $X$. Since the holonomy group $\Hol(X)$ of $X$ with respect to the Kähler metric  is either $\SU(n)$ or $\Sp(n/2)$ where $n = \dim X$, there exists a Kähler-Einstein metric on $T_X$ whose corresponding  Kähler form is cohomologous to $\go$~\cite{Yau}, which further implies that $\gO_X^{k}$ is $\go$-polystable by the Donaldson-Uhlenbeck-Yau theorem~\cite{DonaldsonUYsurface, UhlenbeckYau}.  Precisely, $\gO_X^{k} =  E_1 \oplus \cdots \oplus E_m$ where $E_i$ is $\go$-stable of slope $\mu_\go(E)=0$ and is defined as the parallel transport of each summand in the decomposition of the $\Hol(X)$-module ${\gO_X^{k}}_{|x}$ into irreducible representations over any $x \in X$ described as follows (\emph{cf.}~\cite[\S$13$, $n^o$ $1$ and $3$]{BourbakiLie8} or in~\cite[Chapter VI.3]{BrockerDieck}).

\begin{lem} \hfill
\begin{enumerate}[i)]
\item If $\Hol(X) = \SU(n)$, then the $\Hol(X)$-module ${\gO_X^{k}}_{|x}$ is irreducible.
\item If $\Hol(X) = \Sp(n/2)$, then 
$${\gO_X^{k}}_{|x} = \bigoplus_{k \ge k - 2r \ge 0} \eta_{|x}^r \wedge P^{k-2r},$$
where $\eta$ is a holomorphic symplectic 2-form on $X$ and, $P^{k-2r}$ is an irreducible $\Sp(n/2)$-submodule of $\gO_{X|x}^{k-2r}$. Moreover, 
$$\dim_\bC P^{k-2r} = \binom{2n}{k-2r} - \binom{2n}{k-2r - 2}.$$
\end{enumerate}
\end{lem}
By the above lemma, if  $\Hol(X) = \Sp(n/2)$ and $k$ is odd or $\Hol(X) = \SU(n)$, then  $\dim E_i > 1 = \rank (L)$ for all $i$ (since $0 < k <n$). As $c_1(L)$ is pseudo-effective (so $\mu_\go(L) \ge 0 = \mu_\go(E)$) and $E_i$ is stable, there is no non trivial morphism from $L$ to $E_i$ for all $i$, contradicting  the non-vanishing of $L \to \gO_X^{k}$.
   Finally if $\Hol(X) = \Sp(n/2)$ and if $k$ is even, then $m \ge 2$ and there exists exactly one $i$ such that $\rank(E_i) = 1$. Moreover, $E_i \simeq \cO_X$ and $E_i \hookrightarrow \gO_X^k$ is given by the multiplication by $\eta^{k/2}$. We deduce that if $U$ is a Zariski open subset of $X$ restricted to which $f$ is well-defined, then locally the pullback under $f_{|U}$ of a non-zero holomorphic $k$-form $\ga$ on $f(U)$ is proportional to $\eta^{k/2}$, which contradicts the fact that $\eta$ is non-degenerate. 
\end{proof}

As an immediate consequence,
\begin{cor}\label{cor-classfiblag}
The class of a fiber $\pi^{-1}(t)$ of a Lagrangian fibration $\pi : X \to B$ modulo rational equivalence is independent of $t \in B$. In particular, there exists $\mu \in \bZ \bss\{0\}$ such that $L^n = \mu [F]$ in $\CH^n(X)$ where $L \colonec \pi^*c_1(\cO_B(1))$ for some ample divisor $\cO_B(1)$ over $B$ and $[F]$ is the class of any fiber of $\pi$.
\end{cor}

\section{Construction of constant cycles subvarieties on Lagrangian fibrations}\label{sec-ex}

Let $X$ be a variety. A Zariski locally closed subset $Y$ of $X$ is called \emph{constant cycle} if all points in $Y$ are rationally equivalent in $X$. Note that the property of being constant cycle for a subvariety is \emph{birational} in the following sense:

\begin{lem}\label{lem-openconst}
A subvariety $Y$ of $X$ is constant cycle if and only if there exists a Zariski open subset $U$ of $Y$ such that all points in $U$ are rationally equivalent in $X$.
\end{lem}
\begin{proof}[Proof]
This follows from the fact that every zero-cycle in $Y$ is rationally equivalent to a zero-cycle supported in $U$.
\end{proof}

\begin{lem}\label{lem-Qcar}
Let $Y \subset X$ be a connected Zariski locally closed subset. If the image of the Gysin map $i_*  : \CH_0(Y)_\bQ \to \CH_0(X)_\bQ$ is generated by an element $o_Y$ in $\CH_0(X)_\bQ$, then $Y$  is  constant cycle. In this case, we say that $Y$ is \emph{represented by the zero-cycle $o_Y$}.  If $h^{1,0}(X) = 0$ (e.g. when $X$ is a projective hyper-Kähler manifold), the same conclusion holds without $Y$ being connected.
\end{lem}

\begin{proof}[Proof]
It suffices to show that if every point supported on $Y$ is torsion in $\CH_0(X)$, then $Y$ is a constant cycle subvariety. Let 
$$\ga : Y \hookrightarrow X \to \Alb(X)$$
be the composition of the inclusion map $Y \hookrightarrow X$ with the Albanese map $X \to \Alb(X)$. If the image of $i_* : \CH_0(Y) \to \CH_0(X)$ consists of torsion classes, then by Roitman's theorem~\cite{Roitmantor} the map $\ga$ factorizes through $\CH_0(X)_\tors \simeq \Alb(X)_\tors \hookrightarrow \Alb(X)$ \emph{via} the cycle class map $Y \to \CH_0(X)_\tors$. If $Y$ is connected or $\dim  \Alb(X) = h^{1,0}(X) = 0$, then $\ga$ is constant, hence $Y \to \CH_0(X)_\tors \hookrightarrow \CH_0(X)$ is constant. 
\end{proof}

Now we restrict ourselves to constant cycle subvarieties on projective hyper-Kähler manifolds. Let $X$ be a projective hyper-Kähler manifold of dimension $2n$ and let $\eta$ be a holomorphic symplectic $2$-form on $X$. The following result is a direct consequence of Mumford-Roitman's theorem~\cite[Proposition $10.24$]{VoisinII}:

\begin{pro}\label{pro-isoccs}
If $Y$ is a constant cycle subvariety of $X$, then $Y$ is isotropic for $\eta$. In particular, $\dim Y \le n$ and if $\dim Y = n$, then $Y$ is a Lagrangian constant cycle subvariety.
\end{pro}

The rest of Section~\ref{sec-ex} is devoted to the proof of Theorem~\ref{thm-constr} and Theorem~\ref{thm-indep}.

\begin{proof}[Proof of Theorem~\ref{thm-constr}]
First we  prove the following

\begin{lem}\label{Abfini}
Let $A$ be an abelian variety of dimension $g$ and $H$ an ample divisor on $A$. Let $D \colonec \deg (H^g)$ and choose $o \in A$ to be the origin of $A$ with respect to which $H$ is symmetric. Then for every $x \in A$ the following holds: $H^g = D[x]$ in $\CH_0(A)$ if and only if $x$ is a $D$-torsion point.  In particular, there exist exactly $D^{2g}$ points $x \in A$ such that $H^g = D[x]$ in $\CH_0(A)$.
\end{lem}

\begin{proof}[Proof]

Since $H$ is symmetric with respect to $o$, by Poincaré's  formula~\cite[Corollary $16.5.7$]{BirkLange} $H^g = D[o]$ in $\CH_0(A)$. Let $\alb : A \xrightarrow{\sim} \Alb(A)$ be the Albanese map of $A$ with respect to $o$.  Recall that $\alb$ factorizes through the Deligne cycle class map $\ga : \CH_0(A)_{\hom} \to \Alb(A)$, where $\CH_0(A)_{\hom}$ denotes the subgroup of $\CH_0(A)$ homologous to zero and the morphism $A \to \CH_0(A)_{\hom}$ is given by  $x \mapsto [x] - [o]$. If $H^g = D[x]$ in $\CH_0(A)$, then 
\begin{equation}\label{eqn-albtrans}
D \cdot \alb(x) = \ga(D[x] - D[o]) = \ga(H^g - D[o]) = o.
\end{equation}
in $\Alb(A)$. Conversely if $D \cdot \alb(x) = o$, then $\ga ([x] - [o])$ is $D$-torsion. Since the restriction of $\ga$ to the torsion part of $\CH_0(A)_{\hom}$ is an isomorphism onto the torsion part of $A$~\cite{BlochTorsRoitman, Roitmantor}, we conclude that $[x] - [o]$ is also $D$-torsion. Hence $D[x] = D[o] = H^g$ in $\CH_0(A)$. 
\end{proof}

Let $U \subset B$ be a Zariski open subset of $B$ parametrizing smooth fibers of $\pi$ such that $H_{|\pi^{-1}(b)}$ is ample for any $b \in U$. Set $X_U \colonec \pi^{-1}(U)$. By a standard argument (see for example the proof of~\cite[Theorem $10.19$]{VoisinII}), the relative Hilbert scheme $\mathcal{H}$ over $U$, parametrizing the data of a point $t$ in $U$ and a point $x \in X_t$ such that $D[x_t] = H^g_{|X_t}$ in $\CH_0(X_t)$, is a countable union of irreducible subvarieties of $X_U$.

Let $p : \mathcal{H} \to U$ be the natural projection. Since the $X_t$'s are abelian varieties, $p$ is finite and dominant by Lemma~\ref{Abfini}, so there exists an irreducible component $Z$ of $\cH$ such that ${p}_{|Z}$ is finite and dominant as well. By construction, viewing $Z$ as a subvariety of $X_U$, $Z$ is Zariski locally closed in $X$ of dimension $n$; we define $\gS_{\pi,H}$ as the closure of $Z$ in $X$, which is also of dimension $n$. Finally, for every $x \in Z$, let $j:X_t \hookrightarrow X$ be the inclusion of the fiber of $\pi$ containing $x$, then
$D[x] = (j^*H)^n$ in $\CH_0(X_t)$ thus 
\begin{equation}\label{eqn-suppcan}
D[x] = H^n \cdot [F] = \frac{1}{\deg c_1(\cO_B(1))} H^n \cdot \pi^*c_1(\cO_B(1))^n,
\end{equation} 
for some ample line bundle $\cO_B(1)$ over $B$ where the last equality follows from Corollary~\ref{cor-classfiblag}. Hence $Z$ is a Zariski open subset $U$ of $\gS_{\pi,H}$ whose points are rationally equivalent to a scalar multiple of $H^n \cdot \pi^*c_1(\cO_B(1))^n$. We conclude by Lemma~\ref{lem-Qcar} that $U$ is constant cycle, and then by Lemma~\ref{lem-openconst} that $\gS_{\pi,H}$ is a constant cycle subvariety of $X$. 
\end{proof}

Before we start proving of Theorem~\ref{thm-indep},  let us recall the following result of Matsushita and Voisin which will be useful later. Let $\pi : X \to B$ be a Lagrangian fibration and let $L$ be the pullback of an ample divisor class from the base. Let $j: F \hookrightarrow X$ be the inclusion map of a smooth Lagrangian fiber in $X$.

\begin{lem}[Matsushita~\cite{MatsushitaDefoLagFib} + Voisin~\cite{VoisinstabLag}]\label{lem-matsu}
If $j^* : H^2(X,\bQ) \to H^2(F,\bQ)$ denotes the restriction map and $\mu_{[F]} : H^2(X,\bQ) \to H^{2n + 2}(X,\bQ)$ the cup product with $[F]$, then
$$ \ker \mu_{[F]}  = \ker j^* = \ker q(L,\cdot) $$
where $q : \Sym^2H^2(X,\bQ) \to \bQ$ is the Beauville-Bogomolov-Fujiki form associated to $X$. In particular, the image of $j^* : H^2(X,\bQ) \to H^2(F,\bQ)$ is of rank one.
\end{lem}

\begin{proof}
The first equality is exactly the statement of~\cite[Lemme $1.5$]{VoisinstabLag}. For the second equality, by~\cite[Lemma 2.2]{MatsushitaDefoLagFib}, we have the inclusion $\ker q(L,\cdot) \subset \ker j^*$. It follows that
$$b_2(X) - 1 = \dim_\bQ \ker q(L,\cdot) \le  \dim_\bQ \ker j^* \le b_2(X) - 1$$
where the last inequality results from the non-vanishing of $j^*$ (on any ample divisor class). Hence $\ker q(L,\cdot) = \ker j^*$ for dimensional reasons.
\end{proof}

\begin{proof}[Proof of Theorem~\ref{thm-indep}]

Let $\pi : X \to B$ be a Lagrangian fibration on a polarized hyper-Kähler manifold $(X,H)$ and let $D$ be a divisor in $X$. Let $F$ denote a fiber of $\pi$. If $[D]_{|F} = 0$ in $H^2(F,\bQ)$, then $[D] \cdot [L]^n \in H^{2n+2}(X,\bQ)$ is proportional to  $j_*[D] = 0$. Therefore if $\dim X = 4$, then $D^n \cdot L^n = 0$ in $\CH_0(X)$ by~\cite[Theorem $0.8$]{VoisinMW}. If $[D]_{|F} \ne 0$ in $H^2(F,\bQ)$, then since $j^* : H^2(X,\bQ) \to H^2(F,\bQ)$ is of rank one by Lemma~\ref{lem-matsu}, either $D_{|F}$ or $-D_{|F}$ is ample on $F$. Suppose without loss of generality that $D_{|F}$ is ample on $F$, then by Theorem~\ref{thm-constr}, there exists $d \in \bZ \bss \{0\}$ such that $D^n \cdot F=d[x]$ in $\CH_0(X)$ for any point $x$ in the Lagrangian constant cycle subvariety $\gS_{\pi,D}$. Since $D^n \cdot L^n$ is non-trivially proportional to $D^n \cdot F$ in $\CH_0(X)$ by Corollary~\ref{cor-classfiblag}, this proves $i)$.

To prove $ii)$, let $H_1$ and $H_2$ be two divisor classes such that $H_1$ is ample and $H_2$ satisfies the assumption of Theorem~\ref{thm-constr}. By Lemma~\ref{lem-matsu}, for a smooth fiber $F_b \colonec \pi^{-1}(b)$, there exist $\ga,\gb \in \bZ\bss\{0\}$ such that $(\ga {H_1}  - \gb H_2 )_{|{F}_b}$ is cohomologous to $0$ on $F_b$. It follows that $(\ga {H_1}  - \gb H_2 )_{|{F}_b} $  is cohomologous to $0$ on $F_b$ for all $b \in U$ where $U \subset B$ is the smooth locus of $\pi$. This implies by Lemma~\ref{lem-matsu} that the product $[F_b] \cdot (\ga {H_1}  - \gb H_2)$, hence $L^2 \cdot (\ga {H_1}  - \gb H_2)$, is cohomologous to $0$ on $X$. Since a very general deformation of $X$ preserving the Lagrangian fibration and $H_1, H_2$ satisfies Matsushita's conjecture, we can apply~\cite[Theorem 0.8]{VoisinMW} so that $\ga {H_1} \cdot L^2 = \gb H_2 \cdot L^2$ in $\CH^3(X)\otimes \bQ$. It follows that
$$\ga^2 H_1^2 \cdot L^2 -\gb^2 H_2^2 \cdot L^2 = (\ga {H_1} - \gb H_2 ) \cdot (\ga {H_1} + \gb H_2 ) \cdot L^2 = 0 \ \text{ in } \CH_0(X).$$
Since the zero-cycles supported on the constant cycle subvarieties $\gS_{\pi,H_1}$ and $\gS_{\pi,H_2}$ constructed above are proportional to $H_1^2 \cdot L^2$ and $H_2^2 \cdot L^2$ in $\CH_0(X)$ respectively, the second statement of Theorem~\ref{thm-constr} follows.

Now we prove $iii)$. Let $\pi : X \to B$ and $\pi' : X \to B'$ be two Lagrangian fibrations and let $L \colonec \pi^*c_1(\cO_B(1))$ and $L' \colonec \pi^*c_1(\cO_{B'}(1))$. By Theorem~\ref{thm-constr}, what we need to prove is reduced to the statement that  $H^2 \cdot L^2 $ is proportional to $H^2 \cdot {L'}^2$ in $\CH_0(X)$, so we can assume that $L$ and $L'$ are not proportional, otherwise the proof is finished. Since the restriction of $q$ to $\NS(X)_\bQ$ is of signature $(1,1-\rho(X))$, the restriction of $q$ to the two-dimensional subspace generated by $L$ and $L'$ cannot be zero. Since $q(L,L) = q(L',L') = 0$, this implies that $q(L,L') \ne 0$. By Lemma~\ref{lem-matsu}, we have $j_*L'_{|F} = L' \cdot [F] \ne 0 \in H^6(X,\bQ)$ where $j:F \hookrightarrow X$ is the inclusion of a smooth fiber of $\pi$, so $L'_{|F} \ne 0  \in H^2(F,\bQ)$. Again as $j^* : H^2(X,\bQ) \to H^2(F,\bQ)$ is of rank one, either $L'_{|F}$ or $(-L')_{|F}$ is ample. Similarly either $L_{|F'}$ or $(-L)_{|F'}$ is ample where $F'$ is a smooth fiber of $\pi'$.  However as $L$ and $L'$ are pullbacks of ample classes in $B$ and $B'$ respectively, the classes $(-L)_{|F'}$ and $(-L')_{|F}$ cannot be ample. So necessarily, $L'_{|F}$ and $L_{|F'}$ are  ample. Now the second point of Theorem~\ref{thm-indep} implies that the class of a point in $\gS_{\pi,H}$ equals  the class of a point in $\gS_{\pi,L'}$ in $\CH_0(X)$, which is proportional to $L^2 \cdot L'^2 \ne 0$ by Theorem~\ref{thm-constr}. Similarly the class of a point $\gS_{\pi',H}$ is also proportional to $L'^2 \cdot L^2$ in $\CH_0(X)$, which finishes the proof of $iii)$.
\end{proof}

\section*{Acknowledgement}
This note was written up as a part of the author's PhD thesis at CMLS, École Polytechnique. I would like to thank my thesis advisor C. Voisin for introducing me to this beautiful subject and for valuable discussions. I also thank O. Benoist, L. Fu, Ch. Lehn, and G. Pacienza for interesting remarks and questions, and also the referees for carefully reading the manuscript and suggestions.

\bibliographystyle{plain}
\bibliography{svcc2}

\end{document}